\documentclass[final]{siamltex}
\usepackage[T1]{fontenc}

\usepackage{amssymb,latexsym,amsmath}
\usepackage{mathrsfs}
\usepackage{enumitem}
\usepackage{caption}
\usepackage{comment}
\usepackage{inputenc}
 \usepackage{makeidx}
\usepackage{color}

\allowdisplaybreaks

\newtheorem{remark}{Remark}[section]

\allowdisplaybreaks

\usepackage{epsfig} % for postscript graphics files
\usepackage{verbatim}

\usepackage{amsfonts}
\usepackage{float}
\usepackage{multirow}
\title{Sequential Stochastic Control (Single or Multi-Agent) Problems Nearly Admit Change of Measures with Independent Measurements}

\author{Ian Hogeboom-Burr \and Serdar Y\"{u}ksel\thanks{The authors are with the Dept. of Mathematics and Statistics, Queen's University, Kingston K7L 3N6, ON, Canada, {\tt\small \{15ijhb,yuksel\}@queensu.ca}.}}

\begin{document}
\maketitle

\begin{abstract}
Change of measures has been an effective method in stochastic control and analysis; in continuous-time control this follows Girsanov's theorem applied to both fully observed and partially observed models, in decentralized stochastic control (or stochastic dynamic team theory) this is known as Witsenhausen's static reduction, and in discrete-time classical stochastic control Borkar has considered this method for partially observed Markov Decision processes (POMDPs) generalizing Fleming and Pardoux's approach in continuous-time. This method allows for equivalent optimal stochastic control or filtering in a new probability space where the measurements form an independent exogenous process in both discrete-time and continuous-time and the Radon-Nikodym derivative (between the true measure and the reference measure formed via the independent measurement process) is pushed to the cost or dynamics. However, for this to be applicable, an absolute continuity condition is necessary. This raises the following question: can we perturb any discrete-time sequential stochastic control problem by adding some arbitrarily small additive (e.g. Gaussian or otherwise) noise to the measurements to make the system measurements absolutely continuous, so that a change-of-measure (or static reduction) can be applicable with arbitrarily small error in the optimal cost? That is, are all sequential stochastic (single-agent or decentralized multi-agent) problems $\epsilon$-away from being static reducible as far as optimal cost is concerned, for any $\epsilon > 0$? We show that this is possible when the cost function is bounded and continuous in controllers' actions and the action spaces are convex. We also note that the solution and the cost obtained for the perturbed system is realizable (under a randomized policy) for the original model.  
\end{abstract}

\section{Introduction} 

Reformulation of stochastic dynamic control and filtering problems via a change of measure has proven to be very effective for existence, approximations, filtering, and dynamic programming methods under a variety of contexts. In discrete-time, such a formulation leads to an equivalent system under which measurements are exogenous independent variables. For continuous-time stochastic control, this method has been very popular since the analysis in \cite{benevs1971existence}, to avoid stringent conditions on control policies to satisfy the existence of strong solutions to controlled stochastic differential equations (see also \cite{davis1972information,davis1973dynamic}), which allows the control to be a function of an independent Brownian innovations process. This was utilized by Fleming and Pardoux \cite{FlPa82}, who introduced {\it wide-sense admissible} control policies for classical partially observed stochastic control problems (POMDPs), where first the measures were reduced to an exogenous Brownian process and a control policy was defined to be wide-sense admissible if the control and (exogenous) measurements were independent from future increments of the measurement process (see also \cite{bismut1982partially}, where this was utilized further without separating estimation and control in POMDPs). Borkar \cite{Bor00}, \cite{Bor07} obtained existence results in average cost POMDPs via generalizing Fleming and Pardoux's notion to discrete-time POMDPs, where, once again, first the measurements were made independent and then the control and measurements form a wide-sense admissible process if they are independent from future measurements. In non-linear filtering theory such methods have found applications via the Kushner-Kallianpur-Striebel formula \cite{kushner2014partial}.

%{\it Wide sense admissible policies} introduced by Fleming and Pardoux \cite{FlPa82} and prominently used to establish the existence of optimal solutions for partially observed stochastic control problems. Borkar \cite{Bor00,Bor07,Bor03} (see also Borkar and Budhiraja \cite{BoBu04}) has utilized these policies for a coupling/simulation method to arrive at optimality results for average cost partially observed stochastic control problems. The approach here is to first apply a Girsanov type (see Borkar \cite{Bor00,Bor07} for discrete-time models and Witsenhausen \cite{wit88} for decentralized stochastic control) transformation to decouple the measurements from the system via an absolute continuity condition of measurement variables conditioned on the state, with respect to some reference measure; and then use independence properties: In the discrete-time case, $\{Y_n\}$ is i.i.d. and independent of $X_0$ and the system noise $\{W_n\}$, and $\{U_0,\ldots,U_n,Y_0,\ldots,Y_n\}$ is independent of $\{W_n\}$, $X_0$, and $\{Y_m, m>n\}$, for all $n$. (As cautiously noted in \cite[Section 8]{saldiyukselGeoInfoStructure}, if we only required that $\{U_n,Y_0,\ldots,Y_n\}$ is independent of $\{W_n\}$, $X_0$, and $\{Y_m, m>n\}$, for all $n$, then a counterexample on the validity of the relaxation can be established using results from quantum information theory literature).
%

In decentralized stochastic control this method has allowed for rather mild existence and dynamic programming results \cite{YukselBasarBook,YukselWitsenStandardArXiv,YukselSaldiSICON17,gupta2014existence} and approximation results \cite{saldiyuksellinder2017finiteTeam}. We also refer the reader to \cite{charalambous2016decentralized} for relations with classical continuous-time stochastic control where the relation with Girsanov's classical measure transformation \cite{girsanov1960transforming}, \cite{benevs1971existence} is recognized.

 The common thread for the above change of measure arguments is an absolute continuity condition involving a family of conditional probabilities on the measurement variables (given the past history or state realizations) with respect to a common reference measure. 
 
 For many applications, this is not applicable a priori. In particular, if the measurements are deterministically related to the history of actions and states, such a condition does not apply.
 
 In this paper, we focus on the sequential discrete-time (possibly decentralized) setup and answer the following question: Can we perturb any discrete-time sequential (single-agent or decentralized multi-agent) stochastic control problem by adding some arbitrarily small additive (e.g. Gaussian or otherwise) noise to the measurements to make the measurement kernels absolutely continuous with respect to a fixed reference measure, so that a change-of-measure (also known as a static reduction) can be made with arbitrarily small error in the optimal cost? That is, are all sequential stochastic  problems $\epsilon$-away from being static reducible as far as optimal cost is concerned, for any $\epsilon > 0$? We show that this is possible when the cost function is bounded and continuous in agents' actions and the action spaces are convex.

We note that such a result is not direct and is in fact surprising given the general lack of continuity in the space of information structures under weak convergence (which is indeed the notion considered in our positive result when we will, under weak convergence, approach a dirac delta measure by a sequence of Gaussian measures with decreasing variances): \cite[Section 3.1.1]{YukselOptimizationofChannels} shows that the optimal cost in partially observable stochastic control is {\it not} weakly continuous (though upper semi-continuous) on the space of measurement channels in general, and in fact this negative result also applies to setwise perturbations of measurement channels \cite[Section 3.2.1]{YukselOptimizationofChannels} even when the costs considered are bounded continuous and all spaces considered are compact (see also \cite[Section 4]{hogeboom2021continuity}). To obtain our positive result in this paper, we will consider a cautious construction where Lusin's theorem will be utilized twice, first to approximate functions and second to approximate stochastic kernels with their continuous and weakly continuous counterparts, respectively.

We finally note that there exist two lines of relevant studies on this general subject in the stochastic control literature in both discrete-time and continuous-time, for classical fully observed or partially observed systems and often under stringent conditions. One line of work  \cite{hijab1984asymptotic,baras1988dynamic,baras1982dynamic,heunis1987non,james1994risk,dai1996connections,reddy2021some} studies small-noise limits of stochastic systems (more commonly in the fully observed setup) where large-deviations theoretic results are obtained, which often apriori require stringent regularity or smoothness properties on system (or measurement) dynamics to facilitate such analysis. While our results do not give rates of convergence, our results do not require even continuity of the kernels in the state or actions. 

A second line of research involves vanishing perturbation or vanishing viscosity solutions. For fully observed or partially observed controlled diffusions (or deterministic optimal control) non-degenerate diffusions facilitate powerful stochastic analysis methods (such as convexification, compactification, closedness under weak convergence etc.), and therefore it is a natural question to see what happens when one perturbs an original degenerate system with additive noise, leading to a vanishing viscosity solution \cite{fleming2012deterministic,kushner2014partial,fleming2006controlled} for the limit as the noise vanishes (see e.g. \cite[Chapter 7]{arapostathis2012ergodic} for non-degenerate partially observed diffusion problems where relations with viscosity solutions are emphasized). Related studies, under strong regularity conditions or specific models, include \cite{fleming2012deterministic}, \cite{ciampa2021vanishing} and \cite{bianchini2005vanishing}. Our paper is different in two directions: (i) For such methods often one needs tightness and closedness (under weak topology) conditions on the sets of strategic measures to extract a converging subsequence as the perturbation coefficient vanishes; this is often too much to ask for decentralized stochastic control (see e.g. \cite[Theorem 5.5]{YukselWitsenStandardArXiv}) since conditional independence properties are not closed under weak convergence \cite[Theorem 2.7]{YukselSaldiSICON17}. (ii) Even when the information structure is {\it classical} (i.e., information is expanding with perfect recall), which is the case in fully observed or partially observed single-agent setups, weak continuity of the transition kernel is needed for the closedness of strategic measures under weak convergence \cite{Schal}. 

In our paper, these conditions may not be applicable and are not needed.  

%; the analysis would contribute towards a viscosity theory for partially observed stochastic control. 

%, unlike the approach we present. 
%, often large-deviations theoretic, studies) in addition to establishing more direct relations for controlled diffusions (see e.g. \cite[Chapter 7]{arapostathis2012ergodic} where challenges involving non-degenerate diffusions are emphasized). 

%Often the studies build on the large-deviations principle will be interesting

\section{Sequential Stochastic Control: Witsenhausen's Intrinsic Model}\label{witsenInfoStructureReview}
 
In the following, we present the general (possibly decentralized) discrete-time setup, as well as some further characterizations as laid out by Witsenhausen, termed as {\it the Intrinsic Model} \cite{wit75}. \index{Witsenhausen's Intrinsic Model}\index{Sequential teams}
In this model, any action applied at any given time is regarded as applied by an individual decision maker (DM), who acts only once. One advantage of this model, in addition to its generality, is that the definitions regarding information structures can be compactly described.

Suppose that in a stochastic control system considered below, there is a pre-defined order in which the decision makers act. Such systems are called {\it sequential stochastic control systems} or {\it sequential teams}. In the context of a sequential team, the {\it Intrinsic Model} has three components:

\begin{itemize}
\item A collection of {\it measurable spaces} $\{(\Omega, {\cal F}),
(\mathbb{U}^i,{\cal U}^i), (\mathbb{Y}^i,{\cal Y}^i), i \in {\cal N}\}$, with ${\cal N}:=\{1,2,\cdots,N\}$, specifying the system's distinguishable events, and the control and measurement spaces. Here $N=|{\cal N}|$ is the number of control actions taken, and each of these actions is taken by an individual (different) DM (hence, even a DM with perfect recall can be
regarded as a separate decision maker every time it acts). The pair $(\Omega, {\cal F})$ is a
measurable space (on which an underlying probability may be defined). The pair $(\mathbb{U}^i, {\cal U}^i)$
denotes the measurable space from which the action, $u^i$, of decision maker $i$ is selected. The pair $(\mathbb{Y}^i,{\cal Y}^i)$ denotes the measurable observation/measurement space for {\bf DM}$i$.

\item A {\it measurement constraint} which establishes the connection between the observation variables and the system's distinguishable events. The $\mathbb{Y}^i$-valued observation variables are given by $y^i=\eta^i(\omega,{\bf u}^{[1,i-1]})$, ${\bf u}^{[1,i-1]}=\{u^k, k \leq i-1\}$, $\eta^i$ measurable functions and $u^k$ denotes the action of {\bf DM}$k$. Hence, the information variable $y^i$ induces a $\sigma$-field, $\sigma({\cal I}^i)$ over $\Omega \times \prod_{k=1}^{i-1} \mathbb{U}^k$
\item A {\it design constraint} which restricts the set of admissible $N$-tuple control laws $\underline{\gamma}= \{\gamma^1, \gamma^2, \dots, \gamma^N\}$, also called
{\it designs} or {\it policies}, to the set of all
measurable control functions, so that $u^i = \gamma^i(y^i)$, with $y^i=\eta^i(\omega,{\bf u}^{[1,i-1]})$, and $\gamma^i,\eta^i$ measurable functions. Let $\Gamma^i$ denote the set of all admissible policies for {\bf DM}$i$ and let ${\bf \Gamma} = \prod_{k} \Gamma^k$. In our setup, we will also allow for randomized policies, with independent randomizations, for each DM. 
\end{itemize}

We also introduce a fourth ingredient:
\begin{itemize}
\item A {\it probability measure} $P$ defined on $(\Omega, {\cal F})$ which describes the measures on the random events in the model. 
\end{itemize}

Under the setup for this paper, we will assume $\mathbb{Y}^i \subset \mathbb{R}^{k_i}$, for some $k_i \in \mathbb{Z}_{\geq 1}$, for every $i \in \mathcal{N}$, and we will require admissible policies be defined from $\mathbb{R}^{k_i}$ to $\mathbb{U}^i$. All spaces are assumed to be Borel. 

%\begin{itemize}
%\item A {\it probability measure} $P$ defined on $(\Omega_0, {\cal F})$ which describes the measures on the random events in the model. 
%\end{itemize}

%Under this setup, given a cost function $c: (\Omega_0 \times \mathbb{U}^1 \times \dots \times \mathbb{U}^N) \rightarrow \mathbb{R}$, the decision makers select %their policies to minimize the expected cost.

%We note that the intrinsic model of Witsenhausen gives a set-theoretic characterization of information fields, however, for standard Borel spaces, the model above is equivalent to that of Witsenhausen's. We also note that Witsenhausen has shown that the intrinsic model can be written as a special case of an equivalent model (also due to Witsenhausen \cite{wit88}), which has a similar structure as the one (with functional descriptions on the measurement variables) above with slight differences.

We note that the intrinsic model of Witsenhausen as defined in \cite{wit75} provides a set-theoretic characterization of information fields; however, for standard Borel spaces, the model above is equivalent to that of Witsenhausen's. We will utilize this functional representation in our analysis to follow. We also note that Witsenhausen has shown that the intrinsic model can be written as a special case of an equivalent model (also due to Witsenhausen \cite{wit88}), which has a similar structure as the one above (with functional descriptions on the measurement variables) with only some slight differences.

Under this intrinsic model, (the information structure of) a sequential stochastic control (or team) problem is {\it dynamic} if the
information available to at least one DM is affected by the action of at least one other DM. A decentralized problem is {\it static} if the information available at every decision maker is only affected by exogenous disturbances (Nature) $\omega_0 \in \Omega_0$; that is no decision maker can affect the information for another decision maker.
%That is, an information structure is {\it static} if the information variable ${\cal I}^i$ (viewed as a measurable %function) does not depend on any of the previous control actions in the sense that $\sigma({\cal I}^i)$ is not %affected by the policy $\gamma^k$ for any $k \neq i$ and for all $1 \leq i \leq N$. Otherwise, the information %structure is {\it dynamic}.

Information structures can also be classified as {\it classical}, {\it quasi-classical} or {\it nonclassical}. An Information Structure (IS) $\{y^i, 1 \leq i \leq N \}$ is {\it classical} if $y^i$ contains all of the information available to {\bf DM}$k$ for $k < i$. An IS is {\it quasi-classical} or {\it partially nested}, if whenever $u^k$, for some $k < i$, affects $y^i$ through the measurement function $\eta^i$, $y^i$ contains $y^k$ (that is $\sigma(y^k) \subset \sigma(y^i)$). An IS which is not partially nested is {\it nonclassical}.

Classical stochastic control problems (where a single controller acts repeatedly over time with increasing information), i.e., those that are usually known as MDPs (Markov Decision Processes) or POMDPs (Partially Observed Markov Decision Processes), are easily seen to fit in this general formulation: these are {\it classical}
 according to the classification presented above. 
 
%\section{Non-Classical Information Structures: Proof through a Static Reduction}
Now, consider the team policy \[\underline{\gamma} = \{\gamma^1, \cdots, \gamma^N\}\]
and let a cost function be defined as:
\begin{eqnarray}\label{lossF}
J(\underline{\gamma}) = E^{\underline{\gamma}}[c(\omega_0,{\bf u})] = E[c(\omega_0,\gamma^1(y^1),\cdots,\gamma^N(y^N))],
\end{eqnarray}
for some non-negative measurable loss (or cost) function $c: \Omega_0 \times \prod_k \mathbb{U}^k \to \mathbb{R}_+$. Here, we have the notation ${\bf u}=\{u^t, t \in {\cal N}\}$, and $\omega_0$ may be viewed as the cost function relevant exogenous variable contained in $\omega$. 

\begin{definition}\label{Def:TB1}\index{Optimal team cost}
For a given sequential stochastic team problem with a given information
structure, $\{J; \Gamma^i, i\in {\cal N}\}$, a policy (strategy) $N$-tuple
${\underline \gamma}^*:=({\gamma^1}^*,\ldots, {\gamma^N}^*)\in {\bf \Gamma}$ is
an {\it optimal team policy} if
\begin{equation}J({\underline \gamma}^*)=\inf_{{{\underline \gamma}}\in {{\bf \Gamma}}}
J({{\underline \gamma}})=:J^*. \label{eq:5}
\end{equation} 
The expected cost achieved by this strategy, $J^*$, is the optimal cost.
\end{definition}

In the following, we will denote by bold letters the ensemble of random variables across the DMs; that is ${\bf y}=\{y^i, i=1,\cdots,N\}$ and ${\bf u}=\{u^i, i=1,\cdots,N\}$.

\subsection{Static Reduction of Stochastic Dynamic Problems via Change of Measures}\label{staticReductionDynamicTeams}

Following Witsenhausen \cite[Eqn (4.2)]{wit88}, as reviewed in \cite[Section 3.7]{YukselBasarBook}, we say that two information structures are equivalent if: (i) The policy spaces are equivalent/isomorphic in the sense that policies under one information structure are realizable under the other information structure, (ii) the costs achieved under equivalent policies are identical almost surely, and (iii) if there are constraints in the admissible policies, the isomorphism among the policy spaces preserves the constraint conditions. 

A large class of sequential stochastic control or team problems admit an equivalent information structure which is static. This is called the {\it static reduction} of a dynamic team problem. 

For the results to be presented, under a static reduction, we will have {\it individual measurement variables be independent} of each other as well as $\omega_0$. Witsenhausen refers to such an information structure as {\it independent static} in \cite[Section 2.4(e)]{wit88}. This is not possible in general for every sequential problem which admits a static reduction, for example quasi-classical team problems with linear models \cite{HoChu} do not admit such a further reduction, since the measurements are partially nested. 

Consider (a static or a dynamic) team setting according to the intrinsic model where each {\bf DM}$i$ measures $y^i$
% \[y^i=g_i(\omega_0,\omega_i,y^1,\ldots,y^{i-1},u^1,\ldots,u^{i-1}),\] and the decisions are generated by $u^i=\gamma^i(y^i)$, with $1 \leq i \leq N$. Here $\omega_0,\omega_1,\cdots,\omega_N$ are primitive (exogenous) variables, and $g_i: \Omega_0 \times \mathbb{R}^{k_i} \times \prod_{j = 1}^{i-1} \mathbb{R}^{k_j} \times \prod_{j=1}^{i-1} \mathbb{U}^j \rightarrow \mathbb{Y}^{i}$. 
 We will, for every $1 \leq n \leq N$, view the relation
\begin{eqnarray}\label{ProbForm}
 P(dy^n | \omega_0,y^1,y^2,\cdots,y^{n-1}, u^1,u^2,\cdots,u^{n-1}),
\end{eqnarray} 
  as a (controlled) stochastic kernel (to be defined later), and through standard stochastic realization results (see \cite[Lemma 1.2]{gihman2012controlled} or \cite[Lemma 3.1]{BorkarRealization}), we can represent this kernel in a functional form through 
 \begin{eqnarray}\label{fnForm}
 y^n=g_n(\omega_0,\nu,y^1,y^2,\cdots,y^{n-1},u^1,u^2,\cdots,u^{n-1})
 \end{eqnarray} 
 for some independent $[0,1]$-valued $\nu$ and measurable $g_n$.

This team admits an {\it independent-measurements} reduction provided that the following absolute continuity condition holds: For every $t \in {\cal N}$, there exists a reference probability measure $Q_t$ and a function $f_t$ such that for all Borel $S$:
\begin{eqnarray}\label{staticReduc}
&& P(y^t \in S | \omega_0,u^1,u^2,\cdots,u^{t-1}, y^1,y^2,\cdots,y^{t-1}) \nonumber \\
&& \quad \quad = \int_{S} f_t(y^t,\omega_0,u^1,u^2,\cdots,u^{t-1},y^1,y^2,\cdots,y^{t-1}) Q_t(dy^t),
\end{eqnarray}

We can then write (since the action of each DM is determined by the measurement variables under a given policy)
\begin{eqnarray}
&& P(d\omega_0,d{\bf y}, d{\bf u}) \nonumber \\
&&= P(d\omega_0) \prod_{t=1}^N \bigg(f_t(y^t,\omega_0,u^1,u^2,\cdots,u^{t-1},y^1,y^2,\cdots,y^{t-1}) Q_t(dy^t) 1_{\{\gamma^t(y^t) \in du^t\}}\bigg). \nonumber
\end{eqnarray}
The cost function $J(\underline{\gamma})$ can then be written as
%\[J(\underline{\gamma})= \int P(d\omega_0,d{\bf y})c(\omega_0,{\bf y},{\bf u}).\]
\begin{eqnarray}\label{staticReduc2}
&&J(\underline{\gamma}) \nonumber \\
&&= \int P(d\omega_0) \prod_{t=1}^N (f_t(y^t,\omega_0,u^1,u^2,\cdots,u^{t-1},y^1,y^2,\cdots,y^{t-1}) Q_t(dy^t))  c(\omega_0,{\bf u})
\end{eqnarray}
with $u^k = \gamma^k(y^k)$ for $1 \leq k \leq N$, and where now the measurement variables can be regarded as independent from each other, and also from $\omega_0$, and by incorporating the $\{f_t\}$ terms into $c$, we can obtain an equivalent {\it static team} problem. Hence, the essential step is to appropriately adjust the probability space and the cost function.

\begin{remark} [Change of Measure Formula] Denote the joint probability measure on $(\omega_{0}, u^{1},\dots, u^{N}, y^{1}, \dots, y^{N})$ by $P$, and the probability measure of $\omega_{0}$ by $\mathbb{P}^{0}$. If the preceding absolute continuity condition (\ref{staticReduc}) holds, then (under any admissible policy profile $\gamma^1, \cdots, \gamma^N$) there exists a joint reference probability measure ${\mathbb{Q}}$ on $(\omega_{0}, u^{1},\dots, u^{N}, y^{1}, \dots, y^{N})$ such that the probability measure $P$ is absolutely continuous with respect to ${\mathbb{Q}}$ ($P \ll {\mathbb{Q}}$), so that for every Borel set $A$ in $(\Omega_{0}\times \prod_{i=1}^{N}(\mathbb{U}^{i}\times\mathbb{Y}^{i}))$
\begin{eqnarray}
P(A)= \int_{A} \frac{dP}{d{\mathbb{Q}}}\mathbb{Q}(d\omega_{0}, du^{1},\dots, du^{N}, dy^{1}, \dots, dy^{N})\label{eq:nonrandom},
\end{eqnarray}
where the reference probability measure
\begin{eqnarray}
\mathbb{Q}(d\omega_{0}, du^{1},\dots, du^{N}, dy^{1}, \dots, dy^{N}):=\mathbb{P}^{0}(d\omega_0)\prod_{i=1}^{N} Q^{i}(dy^{i}) 1_{\{\gamma^{i}(y^{i}) \in du^{i}\}},\label{eq:Q}
\end{eqnarray}
leads to a Radon-Nikodym derivative:
\begin{eqnarray}
\frac{dP}{d\mathbb{Q}}(\omega_{0},u^{1}, \dots, u^{1},y^{1}\dots, y^{N})=\prod_{i=1}^{N}f^{i}(y^{i},\omega_{0}, u^{1}, \dots, u^{i-1},y^{1},\dots, y^{i-1})\label{eq:dpdq}.
\end{eqnarray}
\end{remark}

The new cost function may now explicitly depend on the measurement values, such that
\begin{eqnarray}
c_s(\omega_0,{\bf y}, {\bf u}) = c(\omega_0,{\bf u}) \prod_{i=1}^N f_i(y^i,\omega_0,u^1,u^2,\cdots,u^{i-1},y^1,y^2,\cdots,y^{i-1}). \label{c_sDefn}
\end{eqnarray}

Here we can reformulate even a static team to one which is, clearly still static, but now with independent measurements which are also independent from the cost relevant exogenous variable $\omega_0$. 

%Such a condition is in general not restrictive. Indeed, as Witsenhausen notes, a static reduction always holds when the measurement variables take values from a countable set, since a reference measure as in $Q^i$ above can be always constructed on the measurement space $\mathbb{Y}^i$ (e.g., $Q^i(z) = \sum_{j \geq 1} 2^{-j} 1_{\{z = m_j\}}$ where $\mathbb{Y}^i=\{m_j, j \in \mathbb{N}\}$) so that the absolute continuity condition always holds. 

As a common example, consider the additive noise measurement model, for some $f$, \[y^1=f(x) + w,\]
with $w$ a noise variable which admits a probability density function $\eta$ (not necessarily positive everywhere). Then, for any Borel $B$
\[P(y^1 \in B | x) = \int \eta(y^1-f(x)) dy^1 =  \int \frac{\eta(y^1-f(x))}{\kappa(y^1)} \kappa(y^1)dy^1  \]
for any $\kappa$ which is a probability density function that is positive everywhere (such as the Gaussian). Thus, the measurement variable can be assumed to be $y^1 \sim \int_{\cdot} \kappa(y)dy$ which is independent of $x$ and has a measure that admits a probability density function $\kappa$. Note that here one or both of $\eta$ and $\kappa$ can be Gaussian as an important special case. 

This fact will be used for the main results of this paper. Since adding Gaussian noise to each DM's measurement will ensure the model satisfies the absolute continuity condition (\ref{fnForm}), if we can show that the optimal expected cost under a setup where each DM receives a measurement corrupted by an additive Gaussian noise channel (with zero mean) converges to the expected cost under the DMs' original channels as the variance of the noise goes to zero, then (\ref{fnForm}) can be satisfied while only altering the optimal expected cost by an arbitrarily small amount. 

%\section{}
\section{Sequential stochastic control problems are nearly static reducible}

\subsection{Assumptions and supporting results}
Recall from (\ref{ProbForm}) and the preceding discussion that in a dynamic team setup the probability measure on the measurements ${\bf y}$ is not fixed as opposed to the static case. Let, for all $n \in {\cal N}$, 
\[h_n = \{\omega_0,y^1,u^1,\cdots,y^{n-1},u^{n-1},y^n,u^n\},\]
and $p^n(dy^n|h_{n-1}) := P(dy^n | h_{n-1})$ be the transition kernel characterizing the measurements of DM~$n$ according to the intrinsic model. We note that this may be obtained by the relation:
\begin{eqnarray}
&& p^n(y^n \in \cdot | \, \omega_0,y^1,u^1,\cdots,y^{n-1},u^{n-1}) \nonumber \\
&& \, := P\bigg(\eta^n(\omega,{\bf u}^{[1,n-1]}) \in \cdot  \bigg| \, \omega_0,y^1,u^1,\cdots,y^{n-1},u^{n-1}\bigg) \label{kernelDefn2}
%&& \, = P\bigg(g^n(\omega_0,\nu, u^{1},\cdots,u^{n-1}) \in \cdot  \bigg| \, \omega_0,y^1,u^1,\cdots,y^{n-1},u^{n-1}\bigg). \label{kernelDefn2}
\end{eqnarray}

We will allow for an agent's measurements through their channel $p^i$ to be garbled by an additive Gaussian noise channel $\tau^i_m$, which provides the DM with measurement $z^i = y^i + q^i_m$, where $q^i_m \in \mathbb{R}^{k_i}$ is a vector of independent identically distributed $\mathbb{R}$-valued Gaussian noise with mean $0$ and variance $1/m$. This is the measurement that the DM then uses when selecting their action $u^i = \gamma^i(z^i)$ and becomes the measurement that can affect future DMs' measurements through their channels $p^{i+1}, \dots, p^N$. We denote the DM channels under the garbled setup by $p^1\tau^{1}_m, \dots, p^N\tau^{N}_m$. We denote the team-optimal cost (i.e. the expected cost achieved under optimal selection of DM policies) under this setup with $J^*(P, p^1\tau^1_m, \dots, p^N\tau^N_m)$, and the team-optimal cost under the original setup with $J^*(P, p^1, \dots, p^N)$.

%In the following, we show that if each agent's kernel (\ref{kernelDefn2}) is weakly continuous in past measurements and actions, then all sequential dynamic teams are nearly static reducible.

We now introduce the following assumptions (where only A1 and A2 will be needed for our main result):

\begin{itemize}
\item[(A1)] The cost function $c$ is continuous and bounded.

\item[(A2)] Each {\bf DM}$i$'s action space $\mathbb{U}^i$ is a convex subsets of $\mathbb{R}^{n_i}$, for some $n_i \in \mathbb{Z}_{\geq1}$.

\item[(A3)] Each {\bf DM}$i$'s measurement channel given in (\ref{kernelDefn2}) is weakly continuous in the following sense: if $(z^1_m, u^1_m, \dots, z^{i-1}_m, u^{i-1}_m) \rightarrow (z^1, u^1, \dots, z^{i-1}, u^{i-1})$, then \[p^i(y^i \in \cdot | \omega_0, z^1_m, u^1_m, \dots, z^{i-1}_m, u^{i-1}_m) \rightarrow p^i(y^i \in \cdot | \omega_0, z^1, u^1, \dots, z^{i-1}, u^{i-1})\] weakly for all $\omega_0 \in \Omega_0$.
\end{itemize}

For example, through the stochastic realization functional form (\ref{fnForm}) with $y^n=g_n(\omega_0,\nu,z^1,z^2,\cdots,z^{n-1},u^1,u^2,\cdots,u^{n-1})$ for some independent $\nu$ and measurable $g_n$, we have that  if \[g_n(\omega_0,\nu, \cdot): (z^1, u^1, \dots, z^{n-1}, u^{n-1}) \mapsto y^n\]
is continuous, then the kernel would satisfy A3. 

We will show that under assumption A1 and A2
 \[\lim_{m \rightarrow \infty}J^*(P, p^1\tau^1_m, \dots, p^N\tau^N_m) = J^*(P, p^1, \dots, p^N).\]

From this, it follows that any sequential dynamic team problem is $\epsilon$-away in optimal cost from being static reducible. 

First, a result will be stated and proved under the additional assumption A3, in Theorem \ref{weakVer}. Then, this result will be relaxed to include measurement channels that do not necessarily satisfy A3 in Theorem \ref{mainThmNoWeak}, which is the main theorem of our paper. 

The proof will first use Lusin's theorem to construct continuous policies for each DM that approximate their equilibrium policies on a set of $1 - \epsilon_i$, for any $\epsilon_i > 0$. Then, it will be shown using a generalized dominated convergence theorem that the garbled information structure converges weakly to the ungarbled information structure. These steps will then be combined to prove the desired statement. The relaxation in Theorem \ref{mainThmNoWeak} will use Lusin's theorem an additional time to show that each agent's channel can be approximated by a channel that satisfies A3 on a set of measure $1 - \beta_i$, for any $\beta_i > 0$. 

%\section{Main Results}

Our approach is related to perturbing information structures under the weak convergence topology. This problem was studied in \cite{YukselOptimizationofChannels} with negative results presented showing that the optimal cost is not in general weakly continuous on the space of channels. In the following, however, by restricting the nature of allowable perturbations, we will establish a positive result on continuity.  

First we will require an extension of the forward direction to Blackwell's ordering of information structures \cite{Blackwell1951} to sequential dynamic teams in standard Borel spaces (see also \cite[Theorem 4.3.3]{YukselBasarBook}). 
\begin{lemma}\label{GarblingLemma}
For sequential dynamic team problems, if each agent's information under a set of measurement channels $(p^1, \dots, p^N)$ is a garbling of their information under another set of measurement channels $(\Tilde{p}_1, \dots, \Tilde{p}_N)$ then the team-optimal cost cannot be strictly lower under $(p^1, \dots, p^N)$ than under $(\Tilde{p}_1, \dots, \Tilde{p}_N)$. 
\end{lemma}

By garblings, we mean that each agent's measurements under $(p^1, \dots, p^N)$ are such that $y^i_{p^i} = g^i(y^i_{\Tilde{p}_i}, \nu_i)$ for some measurable $g^i$ and independent noise $\nu_i$, for all $i \in \mathcal{N}$, where $y^i_{\Tilde{p}_i}$ is the measurement under the original channel $\Tilde{p}_i$. 

\begin{proof}{\it of Lemma \ref{GarblingLemma}}
We observe that any outcome for a sequential stochastic problem achievable under $(p^1, \dots, p^N)$ is also achievable under $(\Tilde{p}_1, \dots, \Tilde{p}_N)$, since outcomes are determined by policies, which are measurable functions from measurements to actions, and the measurements under $(p^1, \dots, p^N)$ can be simulated using the measurements from $(\Tilde{p}_1, \dots, \Tilde{p}_N)$ and the independent noise variables $\nu_i$, and thus the resulting expected outcomes can also be achieved under $(\Tilde{p}_1, \dots, \Tilde{p}_N)$. For instance, if $\gamma^{1,*}, \dots, \gamma^{N,*}$ form a $\delta$-optimal policy (i.e. are within $\delta > 0$ of minimizing the expected cost) under channels $({p^1}, \dots, p^n)$, then this team policy under the channels $(\Tilde{p}^1, \dots, \Tilde{p}^N)$ can be simulated by selecting actions using the (randomized) policies $u^i = \gamma^{i,*}(g^i(y^i_{\Tilde{p}^i}, \nu_i))$, for all $i \in \mathcal{N}$, achieving the same expected cost. 
\end{proof}

We now recall Lusin's theorem, which will be used throughout the proof of the main theorems. 

\begin{theorem}\cite[Theorem 7.5.2]{Dud02}\label{LusinDudley}[Lusin's Theorem]
Let $(\mathbb{X}, T)$ be any topological space and $\mu$ a finite,
closed regular Borel measure on $\mathbb{X}$. Let $(\mathbb{S}, d)$ be a separable metric space and let $f$ be a Borel-measurable function from $\mathbb{X}$ into $\mathbb{S}$. Then for any $\epsilon > 0$ there is a closed set $F \subset \mathbb{X}$ such that $\mu(\mathbb{X}\setminus F) < \epsilon$ and the restriction of $f$ to $F$ is continuous.
\end{theorem}

Finally we recall Tietze's extension theorem, which will be used in conjunction with Lusin's theorem in the proofs. This will be used to construct a continuous extension of the continuous function defined on $F$ in Theorem \ref{LusinDudley} to $\mathbb{X}$. Note that we consider a general space setup in the following; this is needed as we will study stochastic kernels as probability measure-valued maps.

\begin{theorem}\cite[Theorem 4.1]{dugundji}\label{tietze}[Dugundji]
Let $\mathbb{X}$ be an arbitrary metric space, $A$ a closed subset of $\mathbb{X}$, $L$ a locally convex linear space, and $f: A \rightarrow L$ a continuous map.
Then there exists a continuous function $f_C: \mathbb{X} \rightarrow L$ such that $f_C(a) = f(a) \: \forall a \in A$. Furthermore, the image of $f_C$ satisfies $f_C(\mathbb{X}) \subset$ $[$convex hull of $f(A)]$.
\end{theorem}

\subsection{Weakly Continuous Measurement Kernels}

Now we present the main result of this section, which will be further generalized in the main result of our paper, Theorem \ref{mainThmNoWeak}, to be presented in the next section.

\begin{theorem}\label{weakVer}
Assume assumptions A1, A2, and A3 hold. For any $\epsilon > 0$, any sequential dynamic team problem is $\epsilon$-away in optimal cost from being static reducible.
\end{theorem}

\begin{proof}
\textbf{Step 1.}
Let $p^i(y^i \in \cdot|\omega_0, y^1, u^1, \dots, y^{i-1}, u^{i-1})$ denote each respective DM's measurement channel for $i \in \mathcal{N}$. Let $\tau^1_m, \dots, \tau^N_m$ be independent additive Gaussian noise channels with variance $1/m$ and mean $0$ for each DM. We will use $y^i$ to denote the measurement by agent $i$ through channel $p^i$, and $z^i$ to denote the corrupted measurement of $y^i$ through additive Gaussian noise channel $\tau^i_m$. 

Let $\bar{\gamma}^* = (\gamma^{1,*}, \dots, \gamma^{N,*})$ be a $\delta$-team-optimal policy (where $\delta >0$ is arbitrary) given cost function $c$ \emph{when the Gaussian noise channels are not used}. I.e. $\bar{\gamma}^*$ is a policy that achieves a cost of $J^*(P, p^1, \dots, p^N) + \delta$, or lower. Let $\mu$ denote the joint probability measure on $\Omega_0 \times \mathbb{Y}^1 \times \dots \mathbb{Y}^N$ when the DMs use channels $p^1, \dots, p^n$, and $\nu_m$ denote the measure on $(\Omega_0 \times \mathbb{Y}^1 \times \mathbb{R}^{k_1} \times \dots \mathbb{Y}^N \times \mathbb{R}^{k_N})$ when the DMs use channels $p^1\tau^1_m, \dots, p^N\tau^N_m$. 

We will now apply Lusin's theorem (Theorem \ref{LusinDudley}) for each decision maker to allow for the restriction of each DM's policy $\gamma^{i,*}$ to be continuous on a subset of $\mathbb{Y}^i$ of arbitrarily large marginal measure under policy $\bar{\gamma}^*$.

There is now the question of whether changing the policy affects the marginal on the measurements. This will be handled via a total probability argument: For each decision maker, we will fix all prior DM's policies as their personal policy from the $\delta$-team-optimal strategy $\bar{\gamma}^*$, while all future DMs' policies can be left as arbitrary measurable functions. We require all prior DMs' policies to be fixed in order to fix the measure on {\bf DM}$i$'s measurement space $\mathbb{Y}^i$, since {\bf DM}$i$'s measurement may depend on previous DMs'  actions (and thus policies). Future DMs' policies are irrelevant to the application of Lusin's theorem, as the theorem guarantees that the marginal on the subset of $\mathbb{Y}^i$ where $\gamma^{i,*}$ is not continuous is small, which is unaffected by future DMs' measures on their measurement spaces. 

For the first DM, by Lusin's theorem, there exists a set $A^1$ such that $\gamma^{1,*}$ is continuous on $(\mathbb{Y}^1 \setminus A^1)$ and
\begin{multline}
\int_{\Omega_0 \times A^1 \times \mathbb{Y}^2 \times \dots \times \mathbb{Y}^N}P(d\omega_0)p^1(dy^1|\omega_0)p^2(dy^2|\omega_0, y^1, f^1(y^1))\dots \\ p^N(dy^N|\omega_0, y^1, f^1(y^1), \dots, y^{N-1}, f^{N-1}(y^{N-1})) < \epsilon_1, \nonumber   
\end{multline}
for any measurable $f^1, \dots, f^{N-1}$. 

For the second DM, for $\epsilon_2 > 0$, given fixed policy $\gamma^{1,*}$ for DM 1, there exists a set $A^2$ such that $\gamma^{2,*}$ is continuous on $(\mathbb{Y}^2 \setminus A^2)$ and
\begin{multline}
\int_{\Omega_0 \times \mathbb{Y}^1 \times A^2 \times \mathbb{Y}^3 \times \dots \times \mathbb{Y}^N}P(d\omega_0)p^1(dy^1|\omega_0)p^2(dy^2|\omega_0, y^1, \gamma^{1,*}(y^1))\dots \\ p^N(dy^N|\omega_0, y^1, \gamma^{1,*}(y^1), y^2, f^2(y^2), \dots, y^{N-1}, f^{N-1}(y^{N-1})) < \epsilon_2, \nonumber
\end{multline}
for any measurable $f^2, \dots, f^{N-1}$. 

Continuing in this way, for the $N$th DM, for $\epsilon_N > 0$, given fixed $\gamma^{1,*}, \dots, \gamma^{N-1, *}$, there exists a set $A^N$ such that $\gamma^{N,*}$ is continuous on $(\mathbb{Y}^N \setminus A^N)$ and
\begin{multline}
 \int_{\Omega_0 \times \mathbb{Y}^1 \times \dots \times \mathbb{Y}^{N-1} \times A^N}P(d\omega_0)p^1(dy^1|\omega_0)p^2(dy^2|\omega_0, y^1, \gamma^{1,*}(y^1))\dots \\ p^N(dy^N|\omega_0, y^1, \gamma^{1,*}(y^1), \dots, y^{N-1}, \gamma^{N-1, *}(y^{N-1})) < \epsilon_N. \nonumber
\end{multline}

Furthermore, following Tietze's extension theorem, Theorem \ref{tietze}, it follows that there exist continuous policies $\gamma^1_C, \dots, \gamma^N_C$, such that $\gamma^i_C(x) = \gamma^{i,*}(x), \: \forall x \in (\mathbb{Y}^i \setminus A^i)$ for each $i \in \{1,\cdots, N\}$. \\

\textbf{Step 2.}
Define the following terms:

$h_i := (\omega_0, y^1, \gamma^{1,*}(y^1), \dots, y^{i-1}, \gamma^{i-1, *}(y^{i-1})),$ where $h_1 = (\omega_0).$

$r_i := (\omega_0, z^1, \gamma^{1,*}(z^1), \dots, z^{i-1}, \gamma^{i-1,*}(z^{i-1})),$ where $r_1 = (\omega_0).$

We now note the following, for any fixed $y^1, \omega_0$:
\begin{align*}
   &\lim_{m \rightarrow \infty} \int \tau^1_m(dz^1|y^1)p^2(dy^2|r_2)\dots p^N(dy^N|r_N)\tau^N_m(dz^N|y^N) c(\omega_0, \gamma^1_C(z^1), \dots,\gamma^N_C(z^N)) \\
   & = \lim_{m \rightarrow \infty} \int \tau^1_m(dz^1|y^1)p^2(dy^2|r_2)\dots p^N(dy^N|r_N)\\ &\qquad \qquad \qquad \qquad \qquad \qquad \times \int \tau^N_m(dz^N|y^N) c(\omega_0, \gamma^1_C(z^1), \dots,\gamma^N_C(z^N))
\end{align*}

Letting $G^{N}_{m}(y^N, z^{N-1}, \dots, z^1) := \int \tau^N_m(dz^N|y^N) c(\omega_0, \gamma^1_C(z^1), \dots,\gamma^N_C(z^N))$ (recalling that $\omega_0$ is fixed), we can observe that $G^{N}_{m}$ is a sequence of real-valued and bounded functions, and that if $(y^N_m, z^{N-1}_m, \dots, z^1_m) \rightarrow (y^N, z^{N-1},\dots, z^1)$, then \[G^{N}_{m}(y^N_m, z^{N-1}_m, \dots, z^{1}_m) \rightarrow c(\omega_0, \gamma^{1}_C(z^1),\dots,\gamma^{N-1}_C(z^{N-1}), \gamma^{N}_C(y^N)),\] i.e. $G^{N}_{m}$ continuously converges (as it is defined in \cite{serfozo1982convergence}) to \[G^N(y^N, z^{N-1}, \dots, z^1) := c(\omega_0, \gamma^{1}_C(z^1), \dots,\gamma^{N-1}_C(z^{N-1}), \gamma^{N}_C(y^N)).\] 

The important aspect is that the measurement $z^N$ is now replaced by $y^N$ in the cost function. This follows from the fact that $\tau^N_m(\cdot|y^N)$ converges weakly to the dirac-delta measure at $y^N$ (see e.g. \cite[Example 12.21]{HunterAppliedAnalysis} for a related argument) and that the DMs' policies are continuous and the cost function is continuous and bounded. Continuing, we have:
\begin{align*}
   &\lim_{m \rightarrow \infty} \int \tau^1_m(dz^1|y^1)p^2(dy^2|r_2)\dots p^N(dy^N|r_N)\tau^N_m(dz^N|y^N) c(\omega_0, \gamma^{1}_C(z^1), \dots,\gamma^{N}_C(z^N)) \\
   & = \lim_{m \rightarrow \infty} \int \tau^1_m(dz^1|y^1)p^2(dy^2|r_2)\dots \tau^{N-1}_m(dz^{N-1}|y^{N-1}) \\
   & \qquad \qquad \qquad \qquad \qquad \qquad \times \int p^N(dy^N|r_N)G^N_{m}(y^N, z^{N-1}, \dots, z^1).
\end{align*}

Let $H^{N}_{m}(z^{N-1}, \dots, z^{1}) :=\int p^N(dy^N|r_N)G^{N}_{m}(y^N, z^{N-1}, \dots, z^1)$. We can observe that, since $G^{N}_{m}$ is a sequence of real-valued and bounded functions, $H^{N}_{m}$ is also such a sequence. Furthermore, if $(z^{N-1}_m, \dots, z^1_m) \rightarrow (z^{N-1}, \dots,z^1)$, then $H^{N}_{m}(z^{N-1}_m, \dots, z^{1}_m)$ continuously converges to $H^N(z^{N-1}, \dots,z^1) := \int p^N(dy^N|r_N)G^{N}(y^N, z^{N-1}, \dots, z^1)$ by the weak continuity assumption on the channel $p^N$ and the continuity of $c$ (and thus $G^N_m$).

Continuing in this way, for decreasing $i$, we iteratively define $G^{i}_m$ for $i \in \{1, \dots, N-1\}$ as 
\[G^{i}_m(y^{i}, z^{i-1}, \dots, z^1) := \int \tau^{i}_m(dz^{i}|y^{i}) H^{i-1}_m(z^{i-1}, \dots, z^1)\]
and $H^{i}_m$ for $i \in \{1, \dots, N-1\}$ as
\[H^{i}_m(z^{i-1}, \dots, z^1) := \int p^{i}(dy^{i}|r_{i})G^i_m(y^i_m, z^{i}_m, \dots, z^1_m).\]

We note that each of these functions are real-valued and bounded, and will continuously converge. In particular,
\begin{align*}&G^{i}_m(y^{i}, z^{i-1}, \dots, z^1) \\ &\rightarrow \int \left(\prod_{j = i}^{N-1} p^{j}(dy^j|h_j)\right)c(\omega_0,\gamma^{1}_C(z^1), \dots, \gamma^{i-1}_C(z^{i-1}), \gamma^{i}_C(y^i), \dots, \gamma^{N}_C(y^N)),\end{align*}
and each $H^i_m$ will converge to $\int p^i(dy^i|r_i)G^i(y^{i}, z^{i-1}, \dots, z^1)$.

Thus, eventually we arrive at:
\begin{align*}
   &\lim_{m \rightarrow \infty} \int \tau^1_m(dz^1|y^1)p^2(dy^2|r_2)\dots p^N(dy^N|r_N)\tau^N_m(dz^N|y^N) c(\omega_0, \gamma^{1}_C(z^1), \dots,\gamma^{N}_C(z^N)) \\
   & = \lim_{m \rightarrow \infty} \int \tau^1_m(dz^1|y^1)H^{2}_m(z^1),
\end{align*}
where $H^2_m$ is a continuously converging sequence of continuous and bounded real-valued functions.

Thus, since $\tau^1_m(\cdot|y^1)$ converges weakly to the dirac-delta measure at $y^1$, applying the Generalized Dominated Convergence Theorem (\cite[Theorem 3.5]{serfozo1982convergence} or \cite[Theorem 3.5]{Lan81}), we have that:
\begin{align*}
   &  \lim_{m \rightarrow \infty} \int \tau^1_m(dz^1|y^1) H^{2}_m(z^1)\\
   & = H^{2}(y^1) \\
   & = \int p^2(dy^2|h_2)p^3(dy^3|h_3)\dots p^N(dy^N|h_N) c(x, \gamma^{1}_C(y^1),\dots, \gamma^{N}_C(y^N)).
\end{align*}

%\begin{align*}
%E_{\tau^1_N, \tau^2_N}[c(x, \gamma^{1,*}, \gamma^{2,*})] &= \int_{\Omega_0 \times \mathbb{Y}^1 \times \mathbb{Y}^2} c(x, \gamma^{1,*}(z^1), \gamma^{2,*}(z^2))P(dx)Q^1(dy^1|x)Q^2(dy^2|x, u^1)\tau^1_N(dz^1|y^1)\tau^2_N(dz^2|y^2) \\
%& = \int_{\Omega_0 \times A \times \mathbb{Y}^2} c(x, \gamma^{1,*}(z^1), \gamma^{2,*}(z^2))P(dx)Q^1(dy^1|x)Q^2(dy^2|x, u^1)\tau^1_N(dz^1|y^1)\tau^2_N(dz^2|y^2) \\ &\quad \quad \quad \quad +  \int_{\Omega_0 \times (\mathbb{Y}^1 \setminus A) \times \mathbb{Y}^2} c(x, \gamma^{1,*}(z^1), \gamma^{2,*}(z^2))P(dx)Q^1(dy^1|x)Q^2(dy^2|x, u^1)\tau^1_N(dz^1|y^1)\tau^2_N(dz^2|y^2) \\
%& \leq \int_{\Omega_0 \times \mathbb{Y}^1 \times \mathbb{Y}^2} c(x, \gamma^{1,*}(z^1), \gamma^{2,*}(z^2))P(dx)Q^1(dy^1|x)Q^2(dy^2|x, u^1)\tau^1_N(dz^1|y^1)\tau^2_N(dz^2|y^2) + \epsilon.
%\end{align*}

%Similarly, we can apply Lusin's theorem to get a continuous policy for agent 2, which differs from $\gamma^{2,*}$ only on a set $B \in \mathcal{B}(\mathbb{Y}^2)$ of measure $\alpha$.
\textbf{Step 3.}
Now, we can complete the proof:
\begin{align*}
&\lim_{m\rightarrow \infty}E^{P, p^1, \dots, p^N}_{\tau^1_N,\dots, \tau^N_N}[c(\omega_0, \gamma^1_C, \dots, \gamma^N_C)] \\
&= \lim_{m \rightarrow \infty}\int c(\omega_0, \gamma^1_C(z^1),\dots, \gamma^N_C(z^N)) \\
& \qquad \qquad \qquad \qquad \times P(d\omega_0)p^1(dy^1|r_1)\tau^1_m(dz^1|y^1)\dots p^N(dy^N|r_N)\tau^N_m(dz^N|y^N) \\
& = \lim_{m \rightarrow \infty} \int P(d\omega_0)p^1(dy^1|\omega_0)\int c(\omega_0, \gamma^1_C(z^1), \dots, \gamma^N_C(z^N)) \tau^1_m(dz^1|y^1) \\
& \qquad \qquad \qquad \qquad \times p^2(dy^2|r_2)\tau^2_m(dz^2|y^2)\dots p^N(dy^N|r_N)\tau^N_m(dz^N|y^N)  \\
& = \int c(\omega_0, \gamma^1_C(y^1), \dots, \gamma^N_C(y^N))P(d\omega_0)p^1(dy^1|h_1)p^2(dy^2|h_2)\dots p^N(dy^N|h_N),
\end{align*}
where the final equality holds by the convergence result from Step 2. Furthermore, we have that:
\begin{align*}
\mu(\Omega_0 \times (\mathbb{Y}^1\setminus A_1) \times \dots \times (\mathbb{Y}^N \setminus A_N)) &\geq 1 - \sum_{i = 1}^N \mu_{\mathbb{Y}^i}(A_i) \\
&= 1 - \sum_{i=1}^N \epsilon_i.
\end{align*}

Letting $M:=\|c\|_{\infty} \in \mathbb{R}_{+}$, we have that: 
\begin{align*}
&\int_{\Omega_0 \times \mathbb{Y}^1 \times \dots \times \mathbb{Y}^N} c(\omega_0, \gamma^1_C(y^1), \dots, \gamma^N_C(y^N))P(d\omega_0)p^1(dy^1|h_1)p^2(dy^2|h_2)\dots p^N(dy^N|h_N)\\
&\leq \bigg(\int_{\Omega_0 \times (\mathbb{Y}^1\setminus A_1) \times \dots \times (\mathbb{Y}^N\setminus A_N)} c(\omega_0, \gamma^{1,*}(y^1), \dots, \gamma^{n,*}(y^N)) \\
& \qquad \qquad \qquad \qquad \times P(d\omega_0)p^1(dy^1|h_1)p^2(dy^2|h_2)\dots p^N(dy^N|
h_N)\bigg) + M\sum_{i=1}^{N}\epsilon_i \\
&\leq \bigg(\int_{\Omega_0 \times \mathbb{Y}^1 \times \dots \times \mathbb{Y}^N} c(\omega_0, \gamma^{1,*}(y^1), \dots, \gamma^{n,*}(y^N)) \\
& \qquad \qquad \qquad \qquad \times P(d\omega_0)p^1(dy^1|h_1)p^2(dy^2|h_2)\dots p^N(dy^N|h_N)\bigg) + 2M\sum_{i=1}^{N}\epsilon_i\\
& \leq J^*(P, p^1, \dots, p^N) + 2M\sum_{i=1}^{N}\epsilon_i + \delta. 
\end{align*}

The above holds since on $(\mathbb{Y}^i\setminus A_i)$, $\gamma^i_C$ is equivalent to $\gamma^{i,*}$. 

Since each $\epsilon_i > 0$ and $\delta$ are arbitrary, we have that:
\[\lim_{m\rightarrow \infty}E^{P, p^1, \dots, p^N}_{\tau^1_m,\dots, \tau^N_m}[c(\omega_0, \gamma^1_C, \dots, \gamma^N_C)]\leq J^*(P,p^1,\dots, p^N).\]

Since we know that, for any $m$, $J^*(P,p^1\tau^1_m, \dots, p^N\tau^N_m) \leq E^{P, p^1\tau^1_m,\dots,p^N\tau^N_m}[c(x, \gamma^1_C, \dots, \gamma^N_C)]$ by definition (since $(\gamma^1_C, \dots, \gamma^N_C)$ are not necessarily optimal policies for these measurement channels), it follows that:

$\lim_{m\rightarrow \infty}J^*(P,p^1\tau^1_m, \dots, p^N\tau^N_m) \leq J^*(P,p^1,\dots, p^N)$. 

Applying Lemma \ref{GarblingLemma}, we know that for every $m \in \mathbb{Z}$, 
\[J^*(P,p^1\tau^1_m, \dots,p^N\tau^N_m) \geq J^*(P,p^1, \dots,p^N).\]

Thus we get $\lim_{m \rightarrow \infty}J^*(P,p^1\tau^1_m, \dots,p^N\tau^N_m) = J^*(P,p^1, \dots,p^N)$. This completes the proof. 
\end{proof}

\subsection{Sequential Stochastic Control Problems are Nearly Static Reducible}

In this section, we show that the weak continuity assumption on the channels (A3) can be omitted using an extra approximation step. This is the main result of the paper.

\begin{theorem}\label{mainThmNoWeak}
Assume Assumptions A1 and A2 hold. For any $\epsilon > 0$, any sequential dynamic team problem is $\epsilon$-away in optimal cost from being static reducible.
\end{theorem}

We note that the theorem proof is also constructive: the solution and the cost obtained for the perturbed system is realizable (under a randomized policy) for the original model.

{\bf Proof program.} The proof is essentially the same as the one for the Theorem \ref{weakVer}, except from the addition of a new step, Step 2, to be presented below. %, and the $\beta$ terms that follow in the previous steps as a result. We use the same general setup as Theorem \ref{weakVer}. 

Figure \ref{fig} displays the general proof program. Here $J^{\bar{\gamma}}(\cdot)$ denotes the expected cost when the policies are given specified by $\bar{\gamma}$. (1) holds because a continuous collection $\gamma^1_C, \dots, \gamma^N_C$ can not achieve a cost better than the optimal cost under channels $p^1\tau^1_m, \dots, p^N\tau^N_m$, for any $m$. (2) holds by applying Step 2 below, where $\beta > 0$ is the (arbitrarily small) error term caused by approximating the channels with weakly continuous kernels. (3) holds by Step 1 below, because $\gamma^1_C, \dots, \gamma^N_C$ approximate the $\delta$-optimal policies $\gamma^{1,*}, \dots, \gamma^{N,*}$, except for on a set of small measure, resulting in an arbitrarily small error term $\epsilon > 0$. After being defined, $\gamma^1_C, \dots, \gamma^N_C$ are fixed throughout the proof, and in Step 2, the application of Lusin's theorem to approximate the DM channels uses these fixed continuous policies, as they affect the measures on the DM action spaces. The terms in (4) are related by Lemma \ref{GarblingLemma}, with the relationship being $J^*(P,p^1\tau^1_m, \dots,p^N\tau^N_m) \geq J^*(P,p^1, \dots,p^N)$. \\

\begin{figure}[h]
\begin{center}
\includegraphics[scale = 0.80]{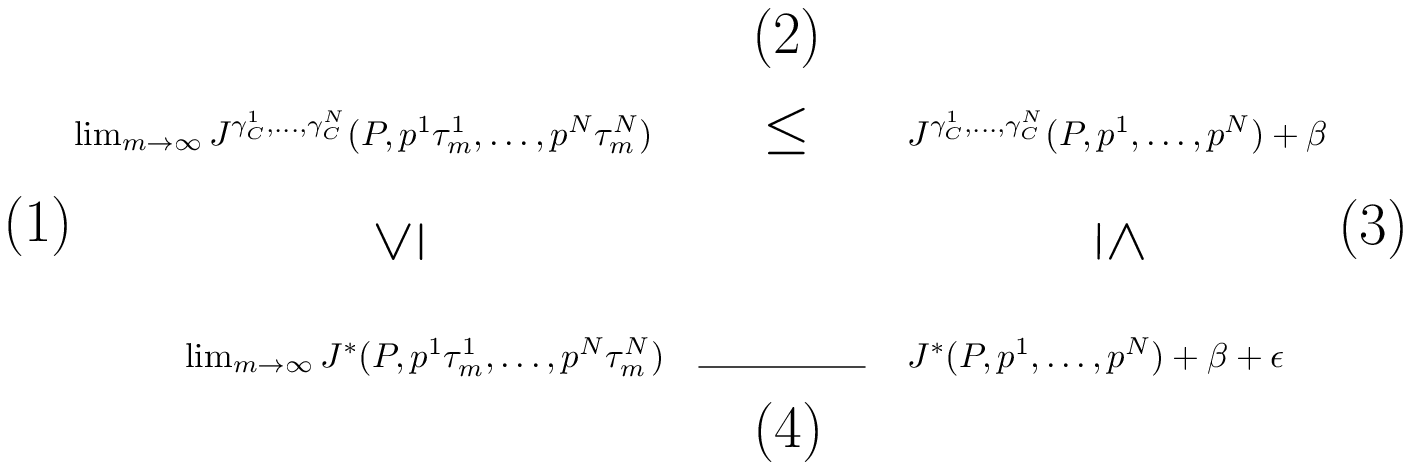} 
\caption{{Proof Program: Arguments of the proof of Theorem \ref{mainThmNoWeak}} \label{fig}}
\end{center}
\end{figure}

We now present the details of the proof.

\begin{proof}{\it (of Theorem \ref{mainThmNoWeak})}

\textbf{Step 1.}
Exactly as in the proof of Theorem \ref{weakVer}, we use Lusin's theorem to construct sets $(\mathbb{Y}^1 \setminus A^1), \dots, (\mathbb{Y}^N \setminus A^N)$ on which $\gamma^{1,*}, \dots, \gamma^{N,*}$ are respectively continuous, and we use Tietze's extension theorem to approximate $\bar{\gamma}^*$ with $\gamma^1_C, \dots, \gamma^N_C$. We will fix this continuous collection of functions in the following.

\textbf{Step 2.}
We now apply Lusin's theorem again, however this time we apply it to each agent's channel to approximate each channel with a weakly continuous kernel. 

Note that a stochastic kernel (as a regular conditional probability) is defined with the following property: for every $(\omega_0, y^1, u^1, \dots, y^{i-1}, u^{i-1})$, $p^i(y^i \in \cdot|\omega_0, y^1, u^1, \dots, y^{i-1}, u^{i-1})$ is $\mathcal{P}(\mathbb{Y}^i)$-valued and for every Borel set $B \subset \mathbb{Y}^i$, 
$p^i(y^i \in B| \cdot): \omega_0, y^1, u^1, \dots, y^{i-1}, u^{i-1} \to \mathbb{R}$ is Borel-measurable.

Following \cite[Proposition 7.26]{BeSh78}, the above property is equivalent to the the kernel $p^i(y^i \in \cdot|\omega_0, y^1, u^1, \dots, y^{i-1}, u^{i-1})$ being a Borel-measurable map from $\Omega_0 \times \mathbb{Y}^1 \times \mathbb{U}^1 \times \dots \times \mathbb{Y}^{i-1} \times \mathbb{U}^{i-1}$ to $\mathcal{P}(\mathbb{Y}^i)$ (which is endowed by the weak convergence topology). 

Because $\mathbb{Y}^i$ is standard Borel, $\mathcal{P}(\mathbb{Y}^i)$ is a separable metric space, and can be defined by viewing the space of probability measures $\mathcal{P}(\mathbb{Y}^i)$ as a convex subset of a locally convex space \cite[Chapter 3]{rudin1991functional} of signed measures defined on ${\cal B}(\mathbb{Y}^i)$, where we define the locally convex space of signed measures with the following notion of convergence: We say that $\nu_n \rightarrow \nu$ if $\int f(y^i) \nu_n(dy^i) \rightarrow \int f(y^i) \nu(dy^i)$ for every continuous and bounded function $f: \mathbb{Y}^i \to \mathbb{R}$. 

Thus, Theorems \ref{LusinDudley} and \ref{tietze} apply with the continuous extension being probability measure-valued (by Theorem \ref{tietze}). 

\sloppy Therefore for any $\beta_i > 0$, for fixed $\gamma^1_C, \dots, \gamma^{i-1}_C$, and $p^1, \dots, p^{i-1}$, there exists a channel $p^i_C$ that satisfies A3, in the sense that if $(\omega_m, y^1_m, u^1_m, \dots, y^{i-1}_m, u^{i-1}_m) \rightarrow (\omega_0, y^1, u^1, \dots, y^{i-1}, u^{i-1})$, then $p^i_C(y^i \in \cdot |\omega_m, y^1_m, u^1_m, \dots, y^{i-1}_m, u^{i-1}_m) \rightarrow p^i_C(y^i \in \cdot|\omega_0, y^1, u^1, \dots, y^{i-1}, u^{i-1})$ weakly. We emphasize here that the DM policies are fixed as the continuous policies developed in Step 1; these policies, in conjunction with fixing the previous DMs' channels, determine the joint probability distribution on $\Omega_0 \times \mathbb{Y}^1 \times \mathbb{U}^1 \times \dots \times \mathbb{Y}^{i-1} \times \mathbb{U}^{i-1}$. Furthermore, $p^i_C(\cdot | \omega_0, y^1, u^1, \dots, y^{i-1}, u^{i-1}) = p^i(\cdot | \omega_0, y^1, u^1, \dots, y^{i-1}, u^{i-1})$ on a set of measure $1 - \beta_i$ on $\Omega_0 \times \mathbb{Y}^1 \times \mathbb{U}^1 \times \dots \times \mathbb{Y}^{i-1} \times \mathbb{U}^{i-1}$. Let $D^i$ denote this set of measure $1 - \beta_i$, and use $D^i_j$ to denote the projection of the set $D^i$ on its $j$th component.

Let $E:= (\cap_{j = 1}^N D^j_1 \times \cap_{j = 2}^N D^j_2 \times \mathbb{R}^{k_1} \times \cap_{j = 2}^N D^j_3 \times \dots \cap_{j = N-1}^N D^{j}_{2N-1} \times D^{N}_{2N} \times \mathbb{R}^{k_N} \times D^N_{2N+1})$. I.e., $E \subset (\Omega_0 \times \mathbb{Y}^1 \times \mathbb{R}^{k_1} \times \mathbb{U}^1 \times \dots \mathbb{Y}^N \times \mathbb{R}^{k_N} \times \mathbb{U}^N)$, where $p^i = p^i_C$ on $E$, for every $i \in \mathcal{N}$, and $E$ is constructed using the intersections of the spaces of arbitrarily large measure on which the channels are weakly continuous. Let $E^c$ denote the complement of $E$ in $(\Omega_0 \times \mathbb{Y}^1 \times \mathbb{R}^{k_1} \times \mathbb{U}^1 \times \dots \mathbb{Y}^N \times \mathbb{R}^{k_N} \times \mathbb{U}^N)$. Let $M:= \|c\|_{\infty} \in \mathbb{R}_+$. Now, applying the fact that the marginal measure on each $D^i$ is known, we note that, given fixed $\gamma^1_C, \dots, \gamma^N_C$, for any $m \in \mathbb{Z}_{\geq 1}$:
\begin{align*}
    &\int_{E^c} c(\omega_0, u^1,\dots, u^N)P(d\omega_0)p^1(dy^1|r_1)\tau^1_m(dz^1|y^1)\gamma^1_C(du^1|z^1)\dots \\
   & \qquad \qquad \qquad \qquad \qquad \qquad \qquad  p^N(dy^N|r_N)\tau^N_m(dz^N|y^N)\gamma^N_C(du^N|z^N) \\
    & \leq M \int_{E^c} P(d\omega_0)p^1(dy^1|r_1)\tau^1_m(dz^1|y^1)\gamma^1_C(du^1|z^1)\dots  \\ & \qquad \qquad \qquad \qquad \qquad \qquad \qquad  p^N(dy^N|r_N)\tau^N_m(dz^N|y^N)\gamma^N_C(du^N|z^N) \\
    & \leq M \sum_{i = 1}^N \beta_i. 
\end{align*}

\textbf{Step 3.}
Assume that $p^1, \dots, p^N$ satisfy A3; this step follows identically to Step 2 from Theorem \ref{weakVer}. Once again, the conclusion is that, for any $\omega_0$ and $y^1$:
\begin{align*}
   &\lim_{m \rightarrow \infty} \int \tau^1_m(dz^1|y^1)p^2(dy^2|r_2)\tau^2_m(dz^2|y^2)\dots  \\
   & \qquad \qquad \qquad \qquad \qquad \qquad \qquad  p^N(dy^N|r_N)\tau^N_m(dz^N|y^N) c(\omega_0, \gamma^{1}_C(z^1), \dots,\gamma^{N}_C(z^N)) \\
   & = \int p^2(dy^2|h_2)p^3(dy^3|h_3)\dots p^N(dy^N|h_N) c(x, \gamma^{1}_C(y^1),\dots, \gamma^{N}_C(y^N)).
\end{align*}

\textbf{Step 4.}
Now, we can show the main result. We have that:
\begin{align*}
&\lim_{m\rightarrow \infty}E^{P, p^1\tau^1_m, \dots, p^N\tau^N_m}[c(\omega_0, \gamma^{1}_C, \dots, \gamma^{N}_C)] \\
&= \lim_{m \rightarrow \infty}\int c(\omega_0, \gamma^{1}_C(z^1),\dots, \gamma^{N}_C(z^N)) \\
   & \qquad \qquad \qquad \times P(d\omega_0)p^1(dy^1|r_1)\tau^1_m(dz^1|y^1)\dots p^N(dy^N|r_N)\tau^N_m(dz^N|y^N) \\
&= \lim_{m \rightarrow \infty}\int c(\omega_0, u^1,\dots, u^N) \\
   & \quad \times P(d\omega_0)p^1(dy^1|r_1)\tau^1_m(dz^1|y^1)\gamma^1_C(du^1|z^1)\dots p^N(dy^N|r_N)\tau^N_m(dz^N|y^N)\gamma^N_C(du^N|z^N) \\
&\leq \bigg(\lim_{m \rightarrow \infty} \int_{E} c(\omega_0, u^1,\dots, u^N)P(d\omega_0)p^1(dy^1|r_1)\tau^1_m(dz^1|y^1)\gamma^1_C(du^1|z^1)\dots \\
& \qquad \qquad \qquad \qquad \qquad \qquad   p^N(dy^N|r_N)\tau^N_m(dz^N|y^N)\gamma^N_C(du^N|z^N)\bigg) + M\sum_{i=1}^N\beta_i\\
&= \bigg(\int_{E}c(\omega_0, u^1,\dots, u^N) P(d\omega_0)p^1(dy^1|h_1)\gamma^1_C(du^1|y^1)p^2(dy^2|h_2)\dots\\
& \qquad \qquad \qquad \qquad \qquad \qquad \qquad \qquad  p^N(dy^N|h_N)\gamma^N_C(du^N|y^N)\bigg) + M\sum_{i=1}^N\beta_i\\
&\leq \bigg(\int c(\omega_0, u^1,\dots, u^N) P(d\omega_0)p^1(dy^1|h_1)\gamma^1_C(du^1|y^1)p^2(dy^2|h_2)\dots\\
& \qquad \qquad \qquad \qquad \qquad \qquad \qquad \qquad  p^N(dy^N|h_N)\gamma^N_C(du^N|y^N)\bigg) + 2M\sum_{i=1}^N\beta_i\\
& =\int c(\omega_0, \gamma^{1}_C(y^1), \dots, \gamma^{N}_C(y^N))P(d\omega_0)p^1(dy^1|h_1)p^2(dy^2|h_2)\dots p^N(dy^N|h_N) \\ & \qquad \qquad \qquad \qquad \qquad \qquad \qquad \qquad \qquad + 2M\sum_{i=1}^N\beta_i\\
&\leq \int c(\omega_0, \gamma^{1,*}(y^1), \dots, \gamma^{N,*}(y^N))P(d\omega_0)p^1(dy^1|h_1)p^2(dy^2|h_2)\dots p^N(dy^N|h_N) \\ &\qquad \qquad \qquad \qquad \qquad \qquad \qquad \qquad  + 2M\sum_{i=1}^{N}\epsilon_i  +   2M\sum_{i=1}^N\beta_i\\
&= J^*(P, p^1, \dots, p^N) + \delta +2M(\sum_{i=1}^{N}\epsilon_i +   \sum_{i=1}^N\beta_i).
\end{align*}

The first several steps follow using Step 2, applying the fact that, if $(z^1, u^1, y^2, z^2, u^2, \dots, y^N, z^N, u^N) \in E$, then, given $\gamma^1_C, \dots, \gamma^N_C$, each $p^i = p^i_C$ on $E$, and each $p^i_C$ satisfies A3. Furthermore, the measure of the complement of $E$ is uniformly small for every $m$. We also use a similar argument as to Theorem \ref{weakVer} to swap the continuous policies with $\bar{\gamma}^*$ for the final inequality. 

Since each $\epsilon_i > 0$, $\beta_i > 0$, and $\delta$ are arbitrary, and $\gamma^1_C, \dots, \gamma^N_C$ are not necessarily optimal policies under $p^1\tau^1_m, \dots, p^N\tau^N_m$, we have that: \[\lim_{m\rightarrow \infty}J^*(P,p^1\tau^1_m, \dots, p^N\tau^N_m) \leq J^*(P,p^1,\dots, p^N).\] 

Applying Lemma \ref{GarblingLemma}, we know that for every $m \in \mathbb{Z}$, 
\[J^*(P,p^1\tau^1_m, \dots,p^N\tau^N_m) \geq J^*(P,p^1, \dots,p^N).\]

Thus we get $\lim_{m \rightarrow \infty}J^*(P,p^1\tau^1_m, \dots,p^N\tau^N_m) = J^*(P,p^1, \dots,p^N)$. This completes the proof. 
\end{proof}

\section{Discussion and Conclusion}
In this paper, we presented results showing that all discrete-time sequential stochastic control (single-agent or multi-agent) problems with continuous and bounded cost functions and convex action spaces are $\epsilon$-away (in optimal cost) from being static reducible. 

The static reduction method has been used in many results to show existence of team-optimal solutions, such as \cite{gupta2014existence}, \cite{YukselSaldiSICON17}, \cite{YukselWitsenStandardArXiv}, \cite{saldiyukselGeoInfoStructure}. Furthermore, \cite{saldiyuksellinder2017finiteTeam} uses the static reduction method to develop results showing that finite models can be used to approximate general team problems with arbitrarily small error in cost. Thus, the results presented here show that these previous existence and approximation results extend (with arbitrarily small error) to all team problems with continuous and bounded cost functions and convex action spaces. 
%\section{Conclusion}

Continuous-time generalizations as well as filtering theoretic applications of our result will be interesting given the importance of non-degeneracy in facilitating mathematical analysis in the theory of non-linear filtering (see \cite{hijab1984asymptotic,reddy2021some,heunis1987non,baras1988dynamic,baras1982dynamic} for related, often large-deviations theoretic, studies) in addition to establishing related results for fully observed or partially observed controlled diffusions (see e.g. \cite[Chapter 7]{arapostathis2012ergodic} where challenges involving non-degenerate diffusions and relations with viscosity solutions are emphasized). Related studies, under strong regularity conditions or specific models, include \cite{ciampa2021vanishing} and \cite{bianchini2005vanishing}). %; the analysis would contribute towards a viscosity theory for partially observed stochastic control. 
\bibliographystyle{siam}
\bibliography{SerdarBibliography}

\begin{thebibliography}{10}

\bibitem{arapostathis2012ergodic}
{\sc A.~Arapostathis, V.~S. Borkar, and M.~K. Ghosh}, {\em Ergodic control of
  diffusion processes}, vol.~143, Cambridge University Press, 2012.

\bibitem{baras1988dynamic}
{\sc J.~S. Baras, A.~Bensoussan, and M.~R. James}, {\em Dynamic observers as
  asymptotic limits of recursive filters: Special cases}, SIAM Journal on
  Applied Mathematics, 48 (1988), pp.~1147--1158.

\bibitem{baras1982dynamic}
{\sc J.~S. Baras and P.~S. Krishnaprasad}, {\em Dynamic observers as asymptotic
  limits of recursive filters}, in 1982 21st IEEE Conference on Decision and
  Control, IEEE, 1982, pp.~1126--1127.

\bibitem{benevs1971existence}
{\sc V.~E. Bene{\v{s}}}, {\em Existence of optimal stochastic control laws},
  SIAM Journal on Control, 9 (1971), pp.~446--472.

\bibitem{BeSh78}
{\sc D.~P. Bertsekas and S.~E. Shreve}, {\em Stochastic optimal control: The
  discrete time case}, Academic Press New York, 1978.

\bibitem{bianchini2005vanishing}
{\sc S.~Bianchini and A.~Bressan}, {\em Vanishing viscosity solutions of
  nonlinear hyperbolic systems}, Annals of mathematics,  (2005), pp.~223--342.

\bibitem{bismut1982partially}
{\sc J.-M. Bismut}, {\em Partially observed diffusions and their control}, SIAM
  Journal on Control and Optimization, 20 (1982), pp.~302--309.

\bibitem{Blackwell1951}
{\sc D.~Blackwell}, {\em The comparison of experiments}, in Proceedings of the
  Second Berkeley Symposium on Mathematical Statistics and Probability,
  (1951), pp.~93--102.

\bibitem{BorkarRealization}
{\sc V.~S. Borkar}, {\em White-noise representations in stochastic realization
  theory}, SIAM J. on Control and Optimization, 31 (1993), pp.~1093--1102.

\bibitem{Bor00}
{\sc V.~S. Borkar}, {\em Average cost dynamic programming equations for
  controlled {M}arkov chains with partial observations}, SIAM J. Control
  Optim., 39 (2000), pp.~673--681.

\bibitem{Bor07}
\leavevmode\vrule height 2pt depth -1.6pt width 23pt, {\em Dynamic programming
  for ergodic control of {M}arkov chains under partial observations: A
  correction}, SIAM J. Control Optim., 45 (2007), pp.~2299--2304.

\bibitem{charalambous2016decentralized}
{\sc C.~D. Charalambous}, {\em Decentralized optimality conditions of
  stochastic differential decision problems via {G}irsanov's measure
  transformation}, Mathematics of Control, Signals, and Systems, 28 (2016),
  pp.~1--55.

\bibitem{ciampa2021vanishing}
{\sc G.~Ciampa and F.~Rossi}, {\em Vanishing viscosity in mean-field optimal
  control}, arXiv preprint arXiv:2111.13015,  (2021).

\bibitem{davis1972information}
{\sc M.~H.~A Davis and P.~Varaiya}, {\em Information states for linear
  stochastic systems}, Journal of Mathematical Analysis and Applications, 37
  (1972), pp.~384--402.

\bibitem{davis1973dynamic}
\leavevmode\vrule height 2pt depth -1.6pt width 23pt, {\em Dynamic programming
  conditions for partially observable stochastic systems}, SIAM Journal on
  Control, 11 (1973), pp.~226--261.

\bibitem{Dud02}
{\sc R.~M. Dudley}, {\em Real Analysis and Probability}, Cambridge University
  Press, Cambridge, 2nd~ed., 2002.

\bibitem{dugundji}
{\sc J.~Dugundji}, {\em An extension of tietze's theorem}, Pacific Journal of
  Mathematics, 1 (1951), pp.~353--367.

\bibitem{FlPa82}
{\sc W.H. Fleming and E.~Pardoux}, {\em Optimal control for partially observed
  diffusions}, SIAM J. Control Optim., 20 (1982), pp.~261--285.

\bibitem{fleming2012deterministic}
{\sc W.~H. Fleming and R.~W. Rishel}, {\em Deterministic and stochastic optimal
  control}, vol.~1, Springer, 2012.

\bibitem{fleming2006controlled}
{\sc W.~H. Fleming and H.~M. Soner}, {\em Controlled Markov processes and
  viscosity solutions}, vol.~25, Springer Science \& Business Media, 2006.

\bibitem{gihman2012controlled}
{\sc I.~I. Gihman and A.~V. Skorohod}, {\em Controlled stochastic processes},
  Springer Science \& Business Media, 2012.

\bibitem{girsanov1960transforming}
{\sc I.~V. Girsanov}, {\em On transforming a certain class of stochastic
  processes by absolutely continuous substitution of measures}, Theory of
  Probability \& Its Applications, 5 (1960), pp.~285--301.

\bibitem{gupta2014existence}
{\sc A.~Gupta, S.~Y\"uksel, T.~Ba\c{s}ar, and C.~Langbort}, {\em On the
  existence of optimal policies for a class of static and sequential dynamic
  teams}, SIAM Journal on Control and Optimization, 53 (2015), pp.~1681--1712.

\bibitem{heunis1987non}
{\sc A.~J. Heunis}, {\em Non-linear filtering of rare events with large
  signal-to-noise ratio}, Journal of applied probability, 24 (1987),
  pp.~929--948.

\bibitem{hijab1984asymptotic}
{\sc O.~Hijab}, {\em Asymptotic bayesian estimation of a first order equation
  with small diffusion}, The Annals of Probability, 12 (1984), pp.~890--902.

\bibitem{HoChu}
{\sc Y.~C. Ho and K.~C. Chu}, {\em Team decision theory and information
  structures in optimal control problems - part {I}}, IEEE Transactions on
  Automatic Control, 17 (1972), pp.~15--22.

\bibitem{hogeboom2021continuity}
{\sc I.~Hogeboom-Burr and S.~Y\"uksel}, {\em Continuity properties of value
  functions in information structures for zero-sum and general games and
  stochastic teams}, arXiv preprint arXiv:2109.11035,  (2021).

\bibitem{HunterAppliedAnalysis}
{\sc J.~Hunter and B.~Nachtergaele}, {\em Applied Analysis}, World Scientific,
  Singapore, 2005.

\bibitem{james1994risk}
{\sc M.~R. James, J.~S. Baras, and R.~E. Elliott}, {\em Risk-sensitive control
  and dynamic games for partially observed discrete-time nonlinear systems},
  IEEE transactions on automatic control, 39 (1994), pp.~780--792.

\bibitem{kushner2014partial}
{\sc H.~J. Kushner}, {\em A partial history of the early development of
  continuous-time nonlinear stochastic systems theory}, Automatica, 50 (2014),
  pp.~303--334.

\bibitem{Lan81}
{\sc H.J. Langen}, {\em Convergence of dynamic programming models}, Math. Oper.
  Res., 6 (1981), pp.~493--512.

\bibitem{dai1996connections}
{\sc P.~Dai Pra, L.~Meneghini, and W.~J. Runggaldier}, {\em Connections between
  stochastic control and dynamic games}, Mathematics of Control, Signals and
  Systems, 9 (1996), pp.~303--326.

\bibitem{reddy2021some}
{\sc A.~S. Reddy, A.~Budhiraja, and A.~Apte}, {\em Some large deviations
  asymptotics in small noise filtering problems}, arXiv preprint
  arXiv:2106.05512,  (2021).

\bibitem{rudin1991functional}
{\sc W.~Rudin}, {\em Functional analysis, mcgraw-hill}, New York,  (1991).

\bibitem{saldiyukselGeoInfoStructure}
{\sc N.~Saldi and S.~Y\"uksel}, {\em Geometry of information structures,
  strategic measures and associated control topologie}, arXiv,  (2020),
  pp.~arXiv--2010.07377.

\bibitem{saldiyuksellinder2017finiteTeam}
{\sc N.~Saldi, S.~Y\"uksel, and T.~Linder}, {\em Finite model approximations
  and asymptotic optimality of quantized policies in decentralized stochastic
  control}, IEEE Transactions on Automatic Control, 62 (2017), pp.~2360 --
  2373.

\bibitem{Schal}
{\sc M.~Sch\"al}, {\em Conditions for optimality in dynamic programming and for
  the limit of n-stage optimal policies to be optimal}, Z.
  Wahrscheinlichkeitsth, 32 (1975), pp.~179--296.

\bibitem{serfozo1982convergence}
{\sc R.~Serfozo}, {\em Convergence of lebesgue integrals with varying
  measures}, Sankhy{\=a}: The Indian Journal of Statistics, Series A,  (1982),
  pp.~380--402.

\bibitem{wit75}
{\sc H.~S. Witsenhausen}, {\em The intrinsic model for discrete stochastic
  control: Some open problems}, Lecture Notes in Econ. and Math. Syst.,
  Springer-Verlag, 107 (1975), pp.~322--335.

\bibitem{wit88}
\leavevmode\vrule height 2pt depth -1.6pt width 23pt, {\em Equivalent
  stochastic control problems}, Math. Control, Signals and Systems, 1 (1988),
  pp.~3--11.

\bibitem{YukselWitsenStandardArXiv}
{\sc S.~Y\"uksel}, {\em A universal dynamic program and refined existence
  results for decentralized stochastic control}, SIAM Journal on Control and
  Optimization, 58 (2020), pp.~2711--2739.

\bibitem{YukselBasarBook}
{\sc S.~Y\"uksel and T.~Ba\c{s}ar}, {\em Stochastic Networked Control Systems:
  Stabilization and Optimization under Information Constraints}, Springer, New
  York, 2013.

\bibitem{YukselOptimizationofChannels}
{\sc S.~Y\"uksel and T.~Linder}, {\em Optimization and convergence of
  observation channels in stochastic control}, SIAM J. on Control and
  Optimization, 50 (2012), pp.~864--887.

\bibitem{YukselSaldiSICON17}
{\sc S.~Y\"uksel and N.~Saldi}, {\em Convex analysis in decentralized
  stochastic control, strategic measures and optimal solutions}, SIAM J. on
  Control and Optimization, 55 (2017), pp.~1--28.

\end{thebibliography}

\end{document}